\documentclass[a4paper, 12pt]{article}
\usepackage[sumlimits]{amsmath}
\usepackage{amsthm}
\usepackage{amssymb}
\usepackage{amsfonts}
\usepackage{graphicx}
\usepackage[latin1]{inputenc}
\usepackage{enumerate}
\usepackage{subfigure}

\newcommand{\bbN}{\mathbb{N}}

\newcommand{\bbC}{\mathbb{C}}

\newcommand{\epi}{\twoheadrightarrow}

\newcommand{\defeq}{\mathrel{\mathop:}=}

\renewcommand{\phi}{\varphi}
\renewcommand{\epsilon}{\varepsilon}

\newcommand{\set}[1]{\ensuremath{\left\{#1\right\}}}

\newcommand{\scr}[1]{\ensuremath{\mathcal{#1}}}
\newcommand{\calC}{\scr C}
\newcommand{\calL}{\scr L}
\newcommand{\catO}{\scr O}
\newcommand{\tcatO}{\widetilde{\catO}}
\newcommand{\calF}{\scr F}
\newcommand{\tcalF}{\widetilde{\scr F}}

\newcommand{\calT}{\scr T}
\newcommand{\tcalT}{\widetilde{\scr T}}

\newcommand{\enva}[1]{\scr U(#1)}
\newcommand{\cntr}[1]{\scr Z(#1)}

\newcommand{\roots}{R}
\newcommand{\rbasis}{\Pi}

\newtheorem{proposition}{Proposition}[section]
\newtheorem{lemma}[proposition]{Lemma}
\newtheorem{theorem}[proposition]{Theorem}
\newtheorem{corollary}[proposition]{Corollary}

\theoremstyle{definition}
\newtheorem{remark}[proposition]{Remark}

\DeclareMathOperator{\Supp}{Supp}
\DeclareMathOperator{\End}{End}
\DeclareMathOperator{\Hom}{Hom}
\DeclareMathOperator{\PFun}{PFun}
\DeclareMathOperator{\TFun}{TFun}
\DeclareMathOperator{\IFun}{IFun}
\DeclareMathOperator{\weight}{w}
\DeclareMathOperator{\Id}{Id}
\DeclareMathOperator{\Tr}{Tr}
\DeclareMathOperator{\im}{im}
\def\clap#1{\hbox to 0pt{\hss#1\hss}}

\def\mathclap{\mathpalette\mathclapinternal}

\def\mathclapinternal#1#2{%
\clap{$\mathsurround=0pt#1{#2}$}}

\newcommand{\restrict}{\delimiter"6223379}
\newcommand{\lowbar}{\underline{\phantom{J}}}
\renewcommand{\frak}{\mathfrak}

\begin{document}

\title{Tensoring with infinite-dimensional\\ modules in $\scr O_0$}

\author{Johan K\aa hrstr\"om}

\maketitle

\abstract{We show that the principal block $\scr O_0$ of the BGG category
	$\scr O$ for a semisimple Lie algebra $\frak g$ acts faithfully on
	itself via exact endofunctors which preserve tilting modules,
	via right exact endofunctors which preserve projective
	modules and via left exact endofunctors which preserve
	injective modules. The origin of all these functors is
	tensoring with arbitrary (not necessarily finite-dimensional)
	modules in the category $\scr O$. We study such functors, describe
	their adjoints and show that they give rise to a natural
	(co)monad structure on $\scr O_0$. Furthermore, all
	this generalises to parabolic subcategories of $\scr O_0$.
	As an example, we present some explicit
	computations for the algebra $\frak{sl}_3$.}

\section{Introduction}
\label{sec:intro}

When studying the category $\catO$ for a semisimple Lie algebra $\frak g$,
tensoring with finite dimensional $\frak g$-modules gives rise to a class of functors
of high importance, the so called projective functors.
These functors were classified in~\cite{bg} and include the
``translation functors'', \cite{jantzen}, which can be used to
prove equivalences of certain subcategories of $\catO$.

In the following we study tensoring with
arbitrary (not necessarily finite dimensional) modules in $\catO$. There is an
immediate obstacle, namely the fact that, in general,
the result is no longer finitely generated (in other words, such functors do not
preserve $\catO$). This can be remedied by projecting
onto a fixed block of the category $\catO$. In particular, by composing
with projection to
the principal block $\catO_0$, we obtain a faithful, exact
functor $G\colon M\mapsto G_M\defeq M\otimes\lowbar\restrict_0$
from $\catO_0$ to the category $\End(\catO_0)$
of endofunctors on $\catO_0$. By, defining $F_M$ and $H_M$ to
be the left and right adjoints of $G_M$,
we obtain a right exact contravariant functor
$F\colon M\mapsto F_M$ and a left exact contravariant functor $H\colon M\mapsto H_M$
from $\catO$ to $\End(\catO_0)$.

In Section~\ref{sec:notation} we introduce the required notions and notation,
and provide a setting for studying the tensor product of arbitrary modules
in $\catO$.
In Section~\ref{sec:main} we define the three functors, and determine some
of their properties. The main properties are given by Theorem~\ref{thm:maintheorem}, which
shows that $F_M$ preserves projectives, $G_M$ preserves tilting modules,
and $H_M$ preserves injectives, for any $M\in\catO_0$.
In Section~\ref{sec:comonad} we show that the particular functors
$G_{\Delta(0)}$ and $G_{\nabla(0)}$ have natural comonad and monad
structures, respectively.
In Section~\ref{sec:parabolic} we show how the results from the previous section
generalize to parabolic subcategories of $\catO$.
Finally, in Section~\ref{sec:example} we compute the `multiplication tables'
$G_{M}N$ and $F_{M}N$ for the case $\frak g=\frak{sl}_3(\bbC)$,
where $M$ and $N$ run over the simple modules in $\catO_0$.

\smallskip
\noindent{\bf Acknowledgments:}
This paper develops some ideas of S.~Ovsienko and V.~Mazorchuk. The author
thanks V.~Mazorchuk for his many comments and suggestions.

\section{Notation and preliminaries}
\label{sec:notation}

For any Lie algebra $\frak a$, we let $\enva{\frak a}$ denote
its universal enveloping algebra.
Fix $\frak g$ to be a finite dimensional complex semisimple Lie algebra,
with a chosen triangular decomposition
$\frak g = \frak n_-\oplus \frak h\oplus \frak n_+$, let
$\frak b = \frak h\oplus \frak n_+$ denote the Borel subalgebra,
and let $\roots$ denote the corresponding root system,
with positive roots $\roots_+$, negative roots $\roots_-$, and
basis $\rbasis$.
Let $\scr O$ denote the corresponding BGG-category (see \cite{bgg} for details),
which can be defined as the full subcategory of the category
of $\frak g$-modules consisting of weight modules
that are finitely generated as $\enva{\frak n_-}$-modules

For a weight module $M$, we denote by $M_\lambda$ the subspace
of $M$ of weight $\lambda\in\frak h^*$, and by
$\Supp M\defeq\{\,\lambda\in\frak h^*\,\vert\,M_\lambda\neq\set{0}\,\}$
the support of $M$. For a weight vector $v\in M$, we denote
by $\weight(v)$ the weight of $v$, i.e. $v\in M_{\weight(v)}$.
Let $\bbN_0$ denote the non-negative integers,
and let $\leqslant$ denote the natural partial order on $\frak h^*$,
i.e. $\lambda\leqslant\mu$ if and only if $\lambda-\mu\in\bbN_0\roots_-$.

Given an anti-automorphism $\theta\colon\frak g\rightarrow \frak g$
of $\frak g$ we define the corresponding restricted duality $d$
on the category of weight $\frak g$-modules as follows. For a weight $\frak g$-module
$M$, let
\[
	dM\defeq \bigoplus_{\lambda\in\frak h^*}\Hom_{\bbC}\bigl(M_\lambda, \bbC\bigr),
\]
with the action of $\frak g$ given by
\[
	(xf)(m) \defeq f\bigl(\theta(x)m\bigr),
\]
for $x\in \frak g$, $f\in dM$ and $m\in M$.

We will use two different
restricted dualities on weight $\frak g$-modules:
the duality given by the anti-automorphism $\frak g\rightarrow \frak g$,
$x\mapsto -x$, which we will denote by $M^*$, and the duality given by
the Chevalley anti-automorphism, which we will denote by $M^\star$.
Note that $\Supp M^\star = \Supp M$, and thus $\star$ preserves
the category $\scr O$, whereas $\Supp M^*=-\Supp M$. `The dual of $M$',
`$M$ is self-dual' and similar statements will, unless otherwise stated,
refer to the $\star$-duality.

Since $\catO$ is not
closed under tensor products (e.g.\ the tensor product of two Verma modules
is never finitely generated and hence does not belong to $\catO$),
it would be convenient to define the
`enlarged' category $\tcatO$, as the
full subcategory of weight $\frak g$-modules $M$ having the properties
\begin{enumerate}
	\item[(OT1)] there are weights $\lambda_1$, $\dotsc$, $\lambda_k\in\frak h^*$
		with \[\Supp M\subseteq \bigcup_{i=1}^k\bigl(\lambda_i+\bbN_0\roots_-\bigr),\]
	\item[(OT2)] $\dim_\bbC M_\lambda<\infty$ for all $\lambda\in\frak h^*$.
\end{enumerate}

\begin{lemma}
	The category $\tcatO$ is closed under tensor products.
\end{lemma}
\begin{proof}
	Let $M, N\in\tcatO$. Then $M\otimes N$ is a weight module,
	and since
	\begin{equation}
		\label{eq:tensorsupp}
		\Supp (M\otimes N) = \Supp M + \Supp N,
	\end{equation}
	it is easy to see that
	the property (OT1) is preserved under
	tensor products. Also,
	\begin{equation}
		\label{eq:dimotimes}
		\dim(M\otimes N)_\lambda
		= \sum_{\mathclap{\substack{\mu\in\Supp M,\\\nu\in\Supp N,\\ \mu+\nu=\lambda}}}
			\dim M_\mu\cdot \dim N_\nu.
	\end{equation}
	By (OT1) the set of pairs
	$\mu\in\Supp M$, $\nu\in\Supp N$ with $\mu+\nu=\lambda$ is finite
	for any $\lambda\in\frak h^*$. By (OT2) we have that
	$\dim M_\mu<\infty$ and $\dim N_\nu<\infty$ for any $\mu$
	and $\nu$, so it follows that the right hand side of~\eqref{eq:dimotimes}
	is finite, i.e. $\dim(M\otimes N)_\lambda<\infty$.
\end{proof}

\begin{lemma}
	\label{lem:dualtens}
	The duality $\star$ commutes with tensor products in $\tcatO$, that is
	\[
		(M\otimes N)^\star\cong M^\star\otimes N^\star,
	\]
	natural in $M$ and $N$.
\end{lemma}

\begin{proof}
	For $f^\star\in M^\star$ and $g^\star\in N^\star$, let
	$\psi(f^\star\otimes g^\star)\in (M\otimes N)^\star$ be defined by
	\[
		\psi(f^\star\otimes g^\star)(m\otimes n) \defeq f^\star(m)g^\star(n),
	\]
	for $m\in M$ and $n\in N$,
	and extended bilinearly to a map $M^\star\otimes N^\star\rightarrow (M\otimes N)^\star$.
	Straightforward verification shows that this is a homomorphism, natural
	in both $M$ and $N$.
	Let $m_1, m_2, \dotsc\in M$ and $n_1, n_2, \dotsc\in N$ be bases of
	weight vectors, and let $m_1^\star, m_2^\star, \dotsc\in M^\star$
	and $n_1^\star, n_2^\star, \dotsc\in N^\star$ be the corresponding
	dual bases. Then we have that $\{\,m_i\otimes n_j\,\vert\,i, j=1, 2, \dotsc\,\}$
	is a basis of $M\otimes N$,
	with the dual basis $\{\,(m_i\otimes n_j)^\star\,\vert\,i, j=1, 2, \dotsc\,\}$. Furthermore,
	\begin{align*}
		\psi(m_i^\star\otimes n_j^\star)(m_k\otimes n_l)
		&= m_i^\star(m_k)n_j^\star(n_l) \\
		&= \delta_{ik}\delta_{jl} \\
		&= (m_i\otimes n_j)^\star(m_k\otimes n_l),
	\end{align*}
	i.e. $\psi(m_i^\star\otimes n_j^\star) = (m_i\otimes n_j)^\star$,
	so $\psi$ is indeed an isomorphism.
\end{proof}

Note that $\catO$ is the full subcategory of $\tcatO$
consisting of finitely generated modules, and in particular
the simple objects of $\tcatO$ and $\catO$ coincide.
For $\lambda\in\frak h^*$, let $L(\lambda)$ denote
the simple highest weight module with highest weight $\lambda$,
and let $P(\lambda)$ denote the projective cover of $L(\lambda)$.

\begin{lemma}
	\label{lem:findecomp}
	All modules $M\in\tcatO$ admit a (possibly infinite)
	composition series. Furthermore, for each $\lambda\in\frak h^*$,
	the number $[M: L(\lambda)]$ of occurrences of $L(\lambda)$
	as a composition factor in a composition series
	is finite and independent of the
	choice of composition series.
\end{lemma}

\begin{proof}
	Let $M\in\tcatO$, and let $m_1$, $m_2$, $m_3$, $\dotsc$, $\in M$
	be a basis of weight vectors such that $\weight(m_i)\leqslant\weight(m_j)$
	implies that $j\leq i$. Such a basis exists due to
	(OT1) and (OT2). For $i\in\bbN_0$, let $M^{(i)}$ denote the
	submodule of $M$ generated by $\set{\,m_j\,\vert\,j\leq i\,}$.
	We thus obtain a series of finitely generated modules
	\[
		\set{0} = M^{(0)} \subseteq M^{(1)}\subseteq M^{(2)}\subseteq M^{(3)}\subseteq\cdots,
	\]
	which, since the $m_i$:s constitute a basis of $M$, converge to $M$, i.e.
	\[
		\bigcup_{i=0}^\infty M^{(i)} = M.
	\]
	Since the $M^{(i)}$:s are finitely generated,
	$M^{(i)}\in\catO$ for all $i\in\bbN_0$. Thus, since all objects in $\catO$
	have finite length, this series can be refined to a composition series.

	Now, consider any composition series $(M^{(i)})$ of $M$,
	let $\lambda\in\frak h^*$ be any weight of $M$, and
	let $N$ denote the submodule of $M$ generated by the weight space $M_\lambda$.
	Since $\dim M_\lambda<\infty$ there exists an index $k\in\bbN$ such that
	$M_\lambda\subseteq M^{(k)}$, and in particular such that $N$ is a
	submodule of $M^{(k)}$.
	Then $\bigl(M^{(i)}/N\bigr)_\lambda=\set{0}$ for all $i\geq k$, so
	\[
		\bigl[(M^{(i)}/N): L(\lambda)\bigr] = 0
	\]
	for all $i\geq k$, and thus
	\[
		[M^{(i)}: L(\lambda)] = [N: L(\lambda)]
	\]
	for all $i\geq k$. As $N\in\catO$, we get that
	$[M: L(\lambda)] = [N: L(\lambda)]$ is finite and independent
	of the choice of composition series.
\end{proof}

Recall that $\catO$ has a block decomposition
\[
	\catO = \bigoplus_{\mathclap{\chi\in \cntr{\frak g}^*}}\catO_\chi,
\]
where $\cntr{\frak g}$ denotes the centre of $\frak g$ and $\catO_\chi$
denotes the full subcategory of $\catO$ consisting of modules $M$
such that for all $z\in \cntr{\frak g}$, $M$ is annihilated by some
power of $\bigl(z-\chi(z)\bigr)$. Hence, each module $M\in\catO$
decomposes into direct sum
\begin{equation}
	\label{eq:odec}
	M = \bigoplus_{\mathclap{\chi\in \cntr{\frak g}^*}}M_\chi,
\end{equation}
where $M_\chi\in\catO_\chi$ and $M_\chi\neq\set{0}$ for only finitely
many $\chi$.

From Lemma~\ref{lem:findecomp} it follows that we get a similar block decomposition
for $\tcatO$, where each module $M\in\tcatO$ decomposes as in~\eqref{eq:odec},
but with possibly countably many non-zero summands (and with some restrictions
on the weight spaces of the non-zero summands). This is similar to
the situation for $\catO$-like categories over a Kac-Moody algebra,
see for example~\cite{neidhardt1, rc-w}. More precisely, we have
the following.

\begin{lemma}
	For all $M\in\tcatO$ and all $\chi\in\cntr{\frak g}^*$ there are
	unique modules (up to isomorphism) $M_1\in\catO_\chi$, $M_2\in\tcatO$, with
	$[M_2:L(\mu)]=0$ for all $\mu\in\frak h^*$ with $L(\mu)\in\catO_\chi$,
	such that
	\[
		M \cong M_1\oplus M_2.
	\]
\end{lemma}

\begin{proof}
	Recall that, for two $\frak g$-modules $K$ and $N$, the
	trace $\Tr_KN$ is defined as the sum of images of all
	homomorphisms from $K$ to $N$. Now, let
	\[
		M_1 \defeq \sum_{\mathclap{\substack{\lambda\in\frak h^*,\\ P(\lambda)\in\catO_\chi}}}
			\Tr_{P(\lambda)}M,
	\]
	and
	\[
		M_2 \defeq
			\sum_{\mathclap{\substack{\lambda\in\frak h^*,\\ P(\lambda)\not\in\catO_\chi}}}
			\Tr_{P(\lambda)}M.
	\]
	As $\catO$ has enough projectives,
	from the proof of Lemma~\ref{lem:findecomp} it follows that
	$M = M_1 + M_2$. Since the central characters occuring in
	$M_2$ are different from $\chi$, this sum must be direct.
\end{proof}

For each $\chi\in \cntr{\frak g}^*$ we thus obtain an exact projection functor
$\lowbar\restrict_\chi\colon\tcatO\rightarrow\catO_\chi$, such that
\begin{equation}
	\label{eq:otdec}
	M = \bigoplus_{\chi\in \cntr{\frak g}^*}M\restrict_\chi
\end{equation}
for any $M\in\tcatO$.

\begin{lemma}
	The tensor product commutes with infinite direct sums in $\tcatO$.
\end{lemma}

\begin{proof}
	Let $N$, $M_1$, $M_2$, $\dotsc\in\tcatO$ with
	\[
		\bigoplus_{i=1}^\infty M_i\in\tcatO,
	\]
	let $n_1$, $n_2$, $\dots\in N$ be a basis of $N$ and let
	$m_1^{(i)}$, $m_2^{(i)}$, $\dots\in M_i$ be a basis of $M_i$
	for each $i\in\bbN$. Then it is immediate that
	\[
		\bigl\{\,m_j^{(i)}\otimes n_k\,\big\vert\,i, j, k\in\bbN\,\bigr\}
	\]
	constitute a basis of both
	\[
		\bigl(M_1\oplus M_2\oplus \dotsb\bigr)\otimes N
	\]
	and
	\[
		(M_1\otimes N)\oplus (M_2\otimes N)\oplus\dotsb,
	\]
	giving the required isomorphism.
\end{proof}

For $\lambda\in\frak h^*$, we denote by $\Delta(\lambda)$ the corresponding
Verma module with highest weight $\lambda$,
and $\nabla(\lambda)\defeq\Delta(\lambda)^\star$ the corresponding dual
Verma module. Let $\calF(\Delta)$ and $\calF(\nabla)$ denote the categories
of modules $M\in\catO$ having a Verma filtration
and dual Verma filtration, respectively,
and let $\calT=\calF(\Delta)\cap\calF(\nabla)$ denote the category of
tilting modules (see~\cite{ringel} for more details).
Let $\tcalF(\Delta)$, $\tcalF(\nabla)$ and $\tcalT$
denote the corresponding categories for $\tcatO$.
As $\star$ commutes with direct sums, the decomposition~\eqref{eq:otdec}
implies that $M\in\tcalF(\Delta)$ if and only if $M^\star\in\tcalF(\nabla)$.
Note also that $\calF(\Delta)$ and $\tcalF(\Delta)$ can be
characterised as the objects in $\catO$ and $\tcatO$
respectively which are free as $\enva{\frak n_-}$-modules.

Similar to
the situation in $\catO$, we have the following result for $\tcatO$
concerning tensor products involving (dual) Verma modules and tilting modules.

\begin{proposition}
	\label{prop:preservestilt}
	For any $M\in\tcatO$, $N\in\tcalF(\Delta)$, $K\in\tcalF(\nabla)$
	and $T\in\tcalT$
	we have $M\otimes N\in\tcalF(\Delta)$, $M\otimes K\in\tcalF(\nabla)$
	and $M\otimes T\in\tcalT$.
\end{proposition}

\begin{proof}
	To show $M\otimes N\in\tcalF(\Delta)$, it suffices to show that
	$M\otimes N\in\tcalF(\Delta)$ for any $N\in\calF(\Delta)$,
	since the general case then follows from the fact that
	any module in $\tcalF(\Delta)$
	decomposes into a direct sum of modules in $\calF(\Delta)$.
	Let $m_1$, $m_2$, $\dotsc$ $\in M$ be a basis of $M$
	constructed as in the proof of Lemma~\ref{lem:findecomp}
	and let $v_1, \dotsc, v_k\in N$ be a basis of $N$ as a $\enva{\frak n_-}$-module
	consisting of weight vectors. We will now show that
	$M\otimes N$ is $\enva{\frak n_-}$-free
 	with the basis $B\defeq\{\,m_i\otimes v_j\,\vert\,i \in\bbN, 1\leq j\leq k\,\}$.
	
 	We start by showing that $B$ generates $M\otimes N$ as a
	$\enva{\frak n_-}$-module. A set that certainly generates $M\otimes N$
	over $\enva{\frak n_-}$ is
	\[
		\bar B\defeq
		\bigl\{\, m_i\otimes(uv_j)\,\big\vert\,i\in\bbN, u\in\enva{\frak n_-},
			1\leq j\leq k\,\bigr\},
	\]
	since $\{m_1, m_2, \dotsc\}$ is a basis of $M$ and
	\[
		\sum_{j=1}^k\{\,uv_j\,\vert\,u\in\enva{\frak n_-}\,\}=N.
	\]
	We will show that $\bar B$ is a subset of the set generated by $B$
	by induction on the degree of $u$.
	So, consider an element $m_i\otimes(uv_j)\in \bar B$. If $u$ has degree $0$,
	then $u$ is a scalar,
	so $m_i\otimes(uv_j)=u(m_i\otimes v_j)$ is in the set generated by $B$.
	Now assume $u$ has degree $d\geq 1$. Then
	\[
		m_i\otimes(uv_j) = u(m_i\otimes v_j) + \sum_{l}(u'_lm_i)\otimes(u''_lv_j),
	\]
 	for some elements elements $u'_l, u''_l\in\enva{\frak n_-}$ with degree strictly
	less than $d$. Since we can rewrite the
	elements $u'_lm_i$ as linear combinations of
	$m_1, m_2, \dotsc$,
	the right hand side is in the set generated by $B$ over $\enva{\frak n_-}$ by the induction
	hypothesis. Hence $B$ generates $\bar B$ as a $\enva{\frak n_-}$-module,
	so $B$ generates $M\otimes N$ as a $\enva{\frak n_-}$-module.
	
	To see that $M\otimes N$ is free over $B$ as a
	$\enva{\frak n_-}$-module, let $L_l$ denote the
	$\enva{\frak n_-}$-submodule of $M\otimes N$ generated by
	\[
		\bigl\{\,m_i\otimes v_j\,\big\vert\,i\in\bbN, i\leq l, 1\leq j\leq k\,\bigr\},
	\]
	and let $\bar L_l$ denote the $\enva{\frak n_-}$-submodule of $M\otimes N$
	generated by
	\[
		\bigl\{\,m_l\otimes v_j\,\vert\,1\leq j\leq k\,\bigr\}.
	\]
	By straightforward induction we see that any non-zero element in $L_l$ has a summand
	of the form $m_i\otimes n$ for some $1\leq i\leq k$, $n\in N$.
	On the other hand, no element of $\bar L_{l+1}$ has such a summand by the ordering of the
	$m_i$:s, and hence we have
	\[
		L_{l+1} = \bar L_{l+1}\oplus L_l.
	\]
	Thus
	\[
		M\otimes N = \bigoplus_{l=1}^\infty \bar L_l
	\]
	as a $\enva{\frak n_-}$-module. Finally, we note that $\bar L_l$ is $\enva{\frak n_-}$-free
 	with the generators
	\[
		\bigl\{\,m_l\otimes v_j\,\vert\,1\leq j\leq k\,\bigr\},
	\]
	since $u(m_l\otimes v_j)$ has a summand of the form
	$m_l\otimes (uv_j)$ for all $u\in\enva{\frak n_-}$. Hence
	$M\otimes N$ is $\enva{\frak n_-}$-free, i.e. $M\otimes N\in\tcalF(\Delta)$.
	
	To show that $M\otimes K\in\tcalF(\nabla)$, note that since
	$K^\star\in\tcalF(\Delta)$, by the previous paragraph we have
	$M^\star\otimes K^\star\in\tcalF(\Delta)$. By
	Lemma~\ref{lem:dualtens}, $\star$
	commutes with tensor products, i.e. $M^\star\otimes K^\star=(M\otimes K)^\star$,
 	and hence $M\otimes K\in\tcalF(\nabla)$.
	
	Finally, since $\tcalT=\tcalF(\Delta)\cap\tcalF(\nabla)$, from the
	first two statements
	it follows that $M\otimes T\in\tcalT$ for all $M\in\tcatO$ and $T\in\tcalT$.
\end{proof}

\begin{corollary}
	For $M\in\tcalF(\Delta)$ and $N\in\tcalF(\nabla)$ we have
	$M\otimes N\in\tcalT$. Furthermore, if $\lambda_1, \lambda_2, \dotsc\in\frak h^*$
	and $\mu_1, \mu_2, \dotsc\in\frak h^*$ are the highest weights, with
	multiplicities,
	of the Verma (respectively dual Verma) modules occurring
	in the Verma and dual Verma filtrations of $M$ and $N$, then
	\[
		M\otimes N \cong \bigoplus_{i, j=1}^\infty\Delta(\lambda_i)\otimes \nabla(\mu_j).
	\]
\end{corollary}

\begin{proof}
	By Proposition~\ref{prop:preservestilt},
	$M\otimes N\in\tcalF(\Delta)\cap\tcalF(\nabla)=\tcalT$.
	Furthermore, $\Delta(\lambda)\otimes\nabla(\mu)\in\tcalT$ for all $\lambda, \mu\in\frak h^*$.
	Since tensoring over a field,
	the second statement now follows from the fact
	that tilting modules do not have
	self-extensions~\cite[Corollary 3]{ringel}.
\end{proof}

Following \cite{fiebig}, for $\lambda\in\frak h^*$ and any weight
module $M$ we define
\[
	M^{\leqslant\lambda}\defeq M/M^{\nleqslant\lambda},
\]
where $M^{\nleqslant\lambda}$ is the submodule of $M$
generated by all the weight spaces $M_\mu$
with $\mu\not\leqslant\lambda$.

\begin{lemma}
	The assignment $\lowbar^{\leqslant\lambda}\colon M\mapsto M^{\leqslant\lambda}$
	defines a right exact functor
	on the category of weight $\frak g$-modules.
\end{lemma}

\begin{proof}
	Let $M$ and $N$ be weight $\frak g$-modules, and let
	$\phi:M\rightarrow N$ be a homomorphism. Since homomorphisms
	preserve weights, the generating set for $M^{\nleqslant\lambda}$
	maps to the generating set for $N^{\nleqslant\lambda}$, and hence
	$\phi\bigl(M^{\nleqslant\lambda}\bigr)\subseteq N^{\nleqslant\lambda}$.
	We thus obtain an induced homomorphism
	\[
		\phi^{\leqslant\lambda}\colon M^{\leqslant\lambda}\rightarrow N^{\leqslant\lambda}.
	\]
	It is immediate that $(\Id_{M})^{\leqslant\lambda} = \Id_{M^{\leqslant\lambda}}$
	and 
	$(\phi\circ\psi)^{\leqslant\lambda}=\phi^{\leqslant\lambda}\circ\psi^{\leqslant\lambda}$,
	so $\lowbar^{\leqslant\lambda}$ is indeed a functor.
	
	Now, consider an exact sequence
	\[
		K\xrightarrow{\psi}M\xrightarrow{\phi}N\rightarrow 0
	\]
	of weight $\frak g$-modules.
	For any element
	$n+N^{\nleqslant\lambda}\in N^{\leqslant\lambda}$, there is an element
	$m\in M$ with $\phi(m)=n$, so
	\[
		\phi^{\leqslant\lambda}\bigl(m+M^{\nleqslant\lambda}\bigr) = n+N^{\nleqslant\lambda},
	\]
	and thus $\phi^{\leqslant\lambda}$ is surjective. Finally, consider an
	element
	\[
		m+M^{\nleqslant\lambda}\in\ker\phi^{\leqslant\lambda},
	\]
	i.e.\ $\phi(m+M^{\nleqslant\lambda})\subseteq N^{\nleqslant\lambda}$.
	Since $\phi$ is surjective, we have $\phi(M^{\nleqslant\lambda})=N^{\nleqslant\lambda}$,
	so there is an element $\tilde m\in M^{\nleqslant\lambda}$ with
	$\phi(\tilde m) = \phi(m)$. Now let $m' = m-\tilde m$. Since
	\[
		\phi(m') = \phi(m)-\phi(\tilde m) = 0
	\]
	we have $m'\in\ker\phi$, and since $\tilde m\in M^{\nleqslant\lambda}$
	we have
	\[
		m'+M^{\nleqslant\lambda}=m+M^{\nleqslant\lambda}.
	\]
	By exactness, there
	is an element $k\in K$ with $\psi(k)=m'$, so
	\[
		\psi^{\leqslant\lambda}\bigl(k+K^{\nleqslant\lambda}\bigr)
		= m'+M^{\nleqslant\lambda} = m+M^{\nleqslant\lambda}.
	\]
	Hence $\im\psi^{\leqslant\lambda}=\ker\phi^{\leqslant\lambda}$, and
	thus $\lowbar^{\leqslant\lambda}$ is right exact.
\end{proof}

\begin{proposition}
	\label{prop:freeproj}
	Let $M$ be an $\enva{\frak n_-}$-free module, say
	\[
		M = \bigoplus_{i\in I}\enva{\frak n_-}v_i
	\]
	as an $\enva{\frak n_-}$-module
	with $\set{\,v_i\,\vert\,i\in I\,}$ being weight vectors.
	Then
	\[
		M^{\leqslant\lambda} \cong
		\bigoplus_{\substack{i\in I,\\ \mathclap{\weight(v_i)\leqslant\lambda}}}
		\enva{\frak n_-}v_i.
	\]
	as a $\enva{\frak n_-}$-module.
\end{proposition}

\begin{proof}
	We claim that
	\[
		M^{\nleqslant\lambda} =
		\sum_{\substack{i\in I,\\ \mathclap{\weight(v_i)\nleqslant\lambda}}}
		\enva{\frak n_-}v_i.
	\]
	To show this, let $N$ denote the set on the right hand side. We need to
	show that $N$ is indeed a submodule of $M$, i.e. closed under the action
	of $\enva{\frak g}$. By the Poincar\'e-Birkhoff-Witt Theorem we know that
	$\enva{\frak g} = \enva{\frak n_-}\,\enva{\frak b}$, so
	\begin{align*}
		\enva{\frak g}N
		&= \sum_{\substack{i\in I,\\ \mathclap{\weight(v_i)\nleqslant\lambda}}}
			\enva{\frak g}\,\enva{\frak n_-}v_i \\
		&= \sum_{\substack{i\in I,\\ \mathclap{\weight(v_i)\nleqslant\lambda}}}
		\enva{\frak g}v_i \\
		&= \sum_{\substack{i\in I,\\ \mathclap{\weight(v_i)\nleqslant\lambda}}}
		\enva{\frak n_-}\,\enva{\frak b}v_i \\
		&\stackrel{(*)}{=} \sum_{\substack{i\in I,\\ \mathclap{\weight(v_i)\nleqslant\lambda}}}
		\enva{\frak n_-}v_i \\
		&= N,
	\end{align*}
	where (*) holds since if $\weight(v_i)\nleqslant\lambda$,
	then $\mu\nleqslant\lambda$ for any $\mu\in\Supp(\enva{\frak b}v_i)$.
	Thus, as a $\enva{\frak n_-}$-module, we have
	\[
		M^{\nleqslant\lambda} =
		\bigoplus_{\substack{i\in I,\\ \mathclap{\weight(v_i)\nleqslant\lambda}}}
		\enva{\frak n_-}v_i,
	\]
	and hence we get that
	\[
		M^{\leqslant\lambda} \cong
		\bigoplus_{\substack{i\in I,\\ \mathclap{\weight(v_i)\leqslant\lambda}}}
		\enva{\frak n_-}v_i.
	\]
	as a $\enva{\frak n_-}$-module.
\end{proof}

\begin{proposition}
	\label{prop:fpreservesdelta}
	For any $M\in\calF(\Delta)$, $N\in\tcatO$ and $\lambda\in\frak h^*$
	we have that $(M\otimes N^*)^{\leqslant\lambda}\in\calF(\Delta)$.
\end{proposition}

\begin{proof}
	Let $m_1, \dotsc, m_k\in M$ be a basis of $M$ as a $\enva{\frak n_-}$-module
	consisting of weight vectors, and let $n_1, n_2, \dotsc\in N$ be a basis
	of $N$ constructed as in the proof of Lemma~\ref{lem:findecomp}.
	By an argument completely analogous to the case where $N$ is finite
	dimensional (see for instance the proof of Theorem~2.2 in \cite{jantzen}),
	it follows that $M\otimes N^*$ is $\enva{\frak n_-}$-free over the set
	\[
		B=\bigl\{\,m_i\otimes n_j^*\,\vert\,1\leq i\leq k, j\in\bbN\,\bigr\}.
	\]
	By Proposition~\ref{prop:freeproj} it follows that $(M\otimes N^*)^{\leqslant\lambda}$
	is $\enva{\frak n_-}$-free, with a $\enva{\frak n_-}$-basis consisting of the
	vectors in $B$ satisfying $\weight(m_i\otimes n_j^*)\leqslant\lambda$. Since
	$N\in\tcatO$, the number of such vectors is finite, and hence
	$(M\otimes N^*)^{\leqslant\lambda}\in\catO$,
	i.e.\ $(M\otimes N^*)^{\leqslant\lambda}\in\calF(\Delta)$.
\end{proof}

\begin{corollary}
	\label{cor:tensdual}
	For each $M\in\catO$, $N\in\tcatO$ and $\lambda\in\frak h^*$
	we have
	\[
		(M\otimes N^*)^{\leqslant\lambda}\in\catO.
	\]
\end{corollary}

\begin{proof}
	Let $P\in\catO$ be the projective cover of $M$. As $\lowbar^{\leqslant\lambda}$
	is right exact, it suffices to prove that $(P\otimes N^*)^{\leqslant\lambda}\in\catO$.
	But this follows from Proposition~\ref{prop:fpreservesdelta},
	since every projective in $\catO$ has a Verma flag.
\end{proof}

\section{The functors}
\label{sec:main}

We now restrict our attention to the principal block $\catO_0$, i.e. the
indecomposable block
containing the trivial module $L(0)$.
Let $\PFun(\catO_0)$, $\TFun(\catO_0)$
and $\IFun(\catO_0)$ denote the categories of endofunctors on $\catO_0$
which preserve the additive subcategories of projective, tilting and
injective modules, respectively. Furthermore, let
$\calF_0(\Delta)=\calF(\Delta)\cap\catO_0$, and define
$\calF_0(\nabla)$ and $\calT_0$ similarly. This section will be devoted
to proving the following theorem, the main result of this paper, along with
some of its consequences.

\begin{theorem}
	\label{thm:maintheorem}
	There exist faithful functors
	\begin{align*}
		F&\colon\catO_0\hookrightarrow\PFun(\catO_0)^{\text{op}}, M\mapsto F_M,\\
		G&\colon\catO_0\hookrightarrow\TFun(\catO_0), M\mapsto G_M,\\
		H&\colon\catO_0\hookrightarrow\IFun(\catO_0)^{\text{op}}, M\mapsto H_M,
	\end{align*}
	all three satisfying $X_M\cong X_N$ if and only if $M\cong N$
	(where $X=F, G, H$).
\end{theorem}

For $M\in\catO_0$, we define the functor $G_M:\catO_0\rightarrow\catO_0$ by
\[
	G_MN \defeq (M\otimes N)\restrict_0
\]
on objects, and
\begin{align*}
	G_M\phi&\colon G_MK\rightarrow G_ML, \\
	G_M\phi&= (\text{Id}_M\otimes \phi)\restrict_0
		\defeq \pi_{G_ML}\circ (\text{Id}_M\otimes \phi) \circ \iota_{G_M K},
\end{align*}
on morphisms $\phi\colon K\rightarrow L$, where
$\pi_{G_ML}\colon M\otimes L\epi (M\otimes L)\restrict_0$ and
$\iota_{G_MK}\colon (M\otimes K)\restrict_0\hookrightarrow M\otimes K$ denote
the natural projection and inclusion. This defines $G_M$ as an endofunctor
on $\catO_0$.

\begin{remark}
	\label{rem:ntransfactors}
	By central character considerations (i.e. from the fact that $G_ML\in\catO_0$), it
	follows that $\pi_{G_ML}\circ(\text{Id}_M\otimes \phi)$ factors through
	$G_M\phi$, i.e. the diagram
	\[
		\includegraphics{ntransfactors.1}
	\]
	commutes.
\end{remark}

For a homomorphism $\phi\colon M\rightarrow N$ between to objects $M, N\in\catO_0$,
we define the corresponding natural transformation $G_{\phi}\colon G_M\rightarrow G_N$
by
\begin{align*}
	G_{\phi}K&\colon G_MK\rightarrow G_NK, \\
	G_{\phi}K&\defeq (\phi\otimes\text{Id}_K)\restrict_0,
\end{align*}
for $K\in\catO_0$. This defines $G$ as a functor from the category $\catO_0$
to the category of endofunctors on $\catO_0$.

Since both $M\otimes\lowbar$ and $\lowbar\restrict_0$ are exact (as the tensor
product is over a field),
it follows that $G_M$ is exact.
Recall that the category $\catO_0$ is equivalent to $A$-mod,
the category of $A$-modules, for some
finite dimensional algebra $A$ (see~\cite{bgg}).
Hence $G_M$ can be seen as an exact functor on $A$-mod,
and in particular $G_M$ is right exact on $A$-mod.
By abstract theory (e.g. Theorem~2.3, \cite{bass}), $G_M$ is naturally
isomorphic to a functor on the form $\overline M\otimes_A\lowbar$
for some $A$-bimodule $\overline M$. We define
\[
	H_M\defeq\Hom_A(\overline M, \lowbar),
\]
the right adjoint of $G_M$. The dual $\star$ is a self-adjoint
contravariant endo\-functor on $\catO_0$, so for any modules $K, L\in\catO_0$
we have the following natural isomorphisms
\begin{align*}
	\Hom_{\catO_0}\bigl(L,(G_{M^\star}K^\star)^\star\bigr)
	&\cong \Hom_{\catO_0}(G_{M^\star}K^\star, L^\star) \\
	&\cong \Hom_{\catO_0}(K^\star, H_{M^\star}L^\star) \\
	&\cong \Hom_{\catO_0}\bigl((H_{M^\star}L^\star)^\star, K\bigr).
\end{align*}
Furthermore, since $\star$ commutes with direct sums and tensor products,
we see that
\[
	(G_{M^\star}K^\star)^\star
	= \bigl((M^\star\otimes K^\star)\restrict_0\bigr)^\star
	= (M\otimes K)\restrict_0 = G_MK.
\]
Thus $\star\circ H_{M^\star}\circ\star$ is the left adjoint of $G_M$, and
we define
\begin{equation}
	\label{eq:fmdef}
	F_M\defeq \star\circ H_{M^\star}\circ\star.
\end{equation}

\begin{proposition}
	For any $M\in\catO_0$ we have that
	$F_M\in\PFun(\catO_0)$, $G_M\in\TFun(\catO_0)$ and
	$H_M\in\IFun(\catO_0)$.
\end{proposition}

\begin{proof}
	That $G_M\in\TFun(\catO_0)$ follows from Proposition~\ref{prop:preservestilt}.
	Assume that $P\in\catO_0$ is projective, i.e. the functor
	$\Hom(P, \lowbar)$ is exact. We need to show that
	$F_MP$ is projective, i.e. that $\Hom(F_MP, \lowbar)$ is exact.
	But
	\[
		\Hom(F_MP, \lowbar)\cong\Hom(P, G_M\lowbar),
	\]
	and the right hand side is the composition of two exact functors,
	so it is exact. The statement $H_M\in\IFun(\catO_0)$ follows by duality.
\end{proof}

\begin{theorem}
	\label{thm:gadj}
	The left adjoint $F_M$ of $G_M$ is given by
	\[
		F_MN = (M^*\otimes N)^{\leqslant 0}\big\restrict_0,
	\]
	and the right adjoint $H_M$ by
	\[
		H_M = \star\circ F_{M^\star}\circ\star.
	\]
\end{theorem}

\begin{proof}
	The second statement follows immediately from the definition \eqref{eq:fmdef}.
	The proof of the first assertion is a slight variation of the proof
	of Proposition~5.1 in \cite{fiebig}, also due to Fiebig. We begin by
	showing that we have a natural isomorphism
	\[
		\Hom_{\frak g}(M^*\otimes K, L)\cong\Hom_{\frak g}(K, M\otimes L)
	\]
	for all $K, L, M\in\catO_0$. Let $m_1, m_2, \dotsc\in M$ be a basis
	consisting of weight vectors, and let $m_1^*, m_2^*, \dotsc\in M^*$
	denote the corresponding dual basis. For $f\in\Hom_{\frak g}(M^*\otimes K, L)$,
	define $\hat f\in\Hom_{\frak g}(K, M\otimes L)$ by
	\[
		\hat f(k) \defeq \sum_{i}m_i\otimes f(m_i^*\otimes k).
	\]
	Since $\Supp L\leqslant 0$, we see that the sum on the right hand side
	is finite, since $f(m_i^*\otimes k)=0$ for all $i$ with
	\[
		\weight(m_i^*\otimes k)\nleqslant 0.
	\]
	For $g\in\Hom_{\frak g}(K, M\otimes L)$, define
	$\tilde g\in\Hom_{\frak g}(M^*\otimes K, L)$ by
	\[
		\tilde g(m_i^*\otimes k)\defeq\sum_jm_i^*(m_j)\cdot l_j,
	\]
	where $g(k)=\sum_jm_j\otimes l_j$ for some weight vectors
	$l_j\in L$ , with $l_j=0$ for almost all $j$.
	The maps $\hat\cdot$ and $\tilde\cdot$
	are indeed inverse to each other, since
	\[
		\tilde{\hat f}(m_i^*\otimes k) = \sum_jm_i^*(m_j)\cdot f(m_j^*\otimes k)
		= f(m_i^*\otimes k),
	\]
	and
	\[
		\hat{\tilde g}(k) = \sum_im_i\otimes \tilde g(m_i^*\otimes k)
		= \sum_{i, j}m_i\otimes \bigl(m_i^*(m_j)\cdot l_j\bigr)
		= \sum_im_i\otimes l_i = g(k),
	\]
	where again $g(k)=\sum_jm_j\otimes l_j$. Hence
	\[
		\Hom_{\frak g}(M^*\otimes K, L)\cong\Hom_{\frak g}(K, M\otimes L),
	\]
	as claimed.
	
	As we saw above, any element $f\in\Hom_{\frak g}(M^*\otimes K, L)$
	is zero on $(M^*\otimes K)^{\nleqslant 0}$, and hence $f$
	factors uniquely through $(M^*\otimes K)^{\leqslant 0}$, so
	\[
		\Hom_{\frak g}\bigl((M^*\otimes K)^{\leqslant 0}, L\bigr)
		\cong \Hom_{\frak g}(M^*\otimes K, L).
	\]
	Also, since $L\in\catO_0$, any element in
	$\Hom_{\frak g}\bigl((M^*\otimes K)^{\leqslant 0}, L\bigr)$
	is zero on any block of $(M^*\otimes K)^{\leqslant 0}$ outside of
	$\catO_0$, so
	\[
		\Hom_{\frak g}\bigl((M^*\otimes K)^{\leqslant 0}, L\bigr)
		\cong
		\Hom_{\frak g}\bigl((M^*\otimes K)^{\leqslant 0}\big\restrict_0, L\bigr).
	\]
	Similarly, since $K\in\catO_0$ we have
	\[
		\Hom_{\frak g}(K, M\otimes L)
		\cong \Hom_{\frak g}(K, M\otimes L\restrict_0).
	\]
	Thus we have obtained a chain of natural isomorphisms
	\begin{align*}
		\Hom_{\frak g}(F_MK, L)
		&= \Hom_{\frak g}\bigl((M^*\otimes K)^{\leqslant 0}\big\restrict_0, L\bigr)
		\cong \Hom_{\frak g}\bigl((M^*\otimes K)^{\leqslant 0}, L\bigr) \\
		&\cong \Hom_{\frak g}(M^*\otimes K, L)
		\cong \Hom_{\frak g}(K, M\otimes L) \\
		&\cong \Hom_{\frak g}(K, M\otimes L\restrict_0)
		=\Hom_{\frak g}(K, G_ML).
	\end{align*}
\end{proof}

\begin{corollary}
	$F$ and $H$ are
	contravariant functors, right and left exact respectively,
	from the category
	$\catO_0$ to the category of endofunctors on $\catO_0$.
\end{corollary}

\begin{proof}
	For $M, N\in\catO_0$, we have by Theorem~\ref{thm:gadj} that
	\[
		F_MN = (M^*\otimes N)^{\leqslant 0}\restrict_0.
	\]
	Analogous to the definition of $G$, for a homomorphism
	$\phi\colon M\rightarrow K$ between objects $M, K\in\catO_0$
	we define the corresponding natural transformation
	$F_\phi\colon F_K\rightarrow F_M$ by
	\[
		F_\phi N\defeq (\phi^*\otimes \text{Id}_N)^{\leqslant 0}\restrict_0
			\colon F_KN
			\rightarrow F_MN.
	\]
	Hence, fixing $N\in\catO_0$, and denoting by $F_{\lowbar}N$ the assignment
	\begin{align*}
		F_{\lowbar}M\colon x&\mapsto F_xM
	\end{align*}
	($x$ being an object or morphism of $\catO_0$),
	we see that
	\[
		F_{\lowbar}N =
		(\lowbar\restrict_0)\circ(\lowbar^{\leqslant0})
		\circ(\lowbar\otimes N)\circ(\lowbar^*).
	\]
	Since $\lowbar^*$ is contravariant exact,
	$\lowbar\otimes N$ is covariant exact,
	$\lowbar^{\leqslant0}$ is covariant right exact,
	and $\lowbar\restrict_0$ is covariant exact, it follows
	that $F_{\lowbar}N$ is a contravariant right exact endofunctor on $\catO_0$,
	which proves the statement for $F$. The statement
	for $H$ follows by duality.
\end{proof}

\begin{remark}
	\label{rem:lzero}
	Note that, since $L(0)^*\cong L(0)\cong{}_{\frak g}\bbC$, with $\frak g$ acting
	trivially on $\bbC$, we have isomorphisms
	\begin{align*}
		G_{L(0)}M &= M\otimes L(0)\restrict_0 \cong M\restrict_0=M, \text{ and} \\
		F_{L(0)}M &= \bigl(M\otimes L(0)^*\bigr)^{\leqslant 0}\big\restrict_0
			\cong M^{\leqslant 0}\big\restrict_0 = M
	\end{align*}
	natural in $M$, for any $M\in\catO_0$. Hence we have natural
	isomorphisms
	\[
		G_{L(0)}\cong F_{L(0)}\cong H_{L(0)}\cong \Id,
	\]
	where $\Id$ denotes the identity functor on $\catO_0$.
\end{remark}

\begin{proposition}
\label{prop:acyclic}
	For any $M\in\catO_0$ the following holds.
	\begin{enumerate}[(a)]
		\item $F_M$ and $G_M$ preserve $\calF_0(\Delta)$ and are
			acyclic on it.
		\item $G_M$ and $H_M$ preserve $\calF_0(\nabla)$ and are
			acyclic on it.
	\end{enumerate}
\end{proposition}

\begin{proof}
	$G_M$ preserves $\calF_0(\Delta)$ and $\calF_0(\nabla)$
	by Proposition~\ref{prop:preservestilt}. $G_M$ is also acyclic on
	$\calF_0(\Delta)$ and $\calF_0(\nabla)$ since $G_M$ is exact.

	$F_M$ preserves $\calF_0(\Delta)$ by Proposition~\ref{prop:fpreservesdelta}.
	If $F_M$ is acyclic on $K$ and $Q$, and the sequence
	\[
		0\rightarrow K\rightarrow N\rightarrow Q\rightarrow 0
	\]
	is exact, then it follows that $F_M$ is acyclic on $N$. Hence
	it suffices to show that $F_M$ is acyclic on Verma modules, by induction
	on the length of Verma flags.
	
	A right exact functor is always acyclic on projective modules, so in
	particular $F_M$ is acyclic on $\Delta(0)$, since $\Delta(0)$ is
	projective. Now let $\lambda\in\frak h^*$ with $\lambda < 0$ and
	$\Delta(\lambda)\in\catO_0$, and assume that $F_M$ is acyclic
	on $\Delta(\mu)$ for all $\mu\in\frak h^*$ with $\lambda < \mu$
	and $\Delta(\mu)\in\catO_0$. All Verma modules fit in a short exact
	sequence
	\begin{equation}
		\label{eq:prop:acyclic}
		0\rightarrow K\rightarrow P(\lambda)\rightarrow\Delta(\lambda)\rightarrow 0,
	\end{equation}
	where $P(\lambda)$ is projective, and $K\in\calF_0(\Delta)$ is
	filtered by Verma modules $\Delta(\mu)$ with $\lambda<\mu$. In particular,
	$F_M$ is acyclic on $K$ by the induction hypothesis. Hence, in the induced
	long exact sequence
	\[
		\dotsb\rightarrow \calL_{i+1}F_MP(\lambda)
		\rightarrow\calL_{i+1}F_M\Delta(\lambda)
		\rightarrow\calL_{i}F_MK
		\rightarrow\dotsb
	\]
	we have $\calL_{i+1}F_MP(\lambda)=0$ since $P(\lambda)$ is projective,
	and $\calL_iF_MK=0$ by the induction hypothesis, so $\calL_{i+1}F_M\Delta(\lambda)=0$
	for all $i>1$. It remains to show that $\calL_1F_M\Delta(\lambda)=0$.
	
	Since $\calL_1F_MP(\lambda)=0$, we have that
	\begin{equation}
		\label{eq:lonewhat}
		0\rightarrow\calL_1F_M\Delta(\lambda)
		\rightarrow F_MK\rightarrow F_MP(\lambda)\rightarrow\Delta(\lambda)\rightarrow 0
	\end{equation}
	is exact. Now, consider the short exact sequence
	\[
		0\rightarrow M^*\otimes K\rightarrow M^*\otimes P(\lambda)
		\rightarrow M^*\otimes \Delta(\lambda)\rightarrow 0
	\]
	obtained from \eqref{eq:prop:acyclic} by applying the functor
	$M^*\otimes \lowbar$. The modules in the above sequence
	are all $\enva{\frak n_-}$-free, so by Proposition~\ref{prop:freeproj}
	we obtain an exact sequence
	\begin{equation}
		\label{eq:lonezero}
		0\rightarrow F_MK\rightarrow F_MP(\lambda)\rightarrow F_M\Delta(\lambda)\rightarrow 0
	\end{equation}
	by applying $\lowbar^{\leqslant 0}$, and thus
	$\calL_1F_M\Delta(\lambda)=0$, by comparing~\eqref{eq:lonewhat} and~\eqref{eq:lonezero}.
	
	Since $H_M=\star\circ F_{M^\star}\circ\star$, and $\star$ is
	a contravariant exact functor swapping $\calF_0(\nabla)$ with
	$\calF_0(\Delta)$, it follows by the dual argument to the previous
	paragraph that $H_M$ preserves $\calF_0(\nabla)$ and is acyclic
	on it.
\end{proof}

\begin{lemma}
	The functors $F$, $G$ and $H$ are faithful.
\end{lemma}

\begin{proof}
	Let $M, N\in\catO_0$ with a non-zero
	homomorphism $\phi\colon M\rightarrow N$. By the symmetry of the
	tensor product, we have	$G_{\lowbar}K\cong G_K\lowbar$ for any $K\in\catO_0$.
	In particular, it follows from Remark~\ref{rem:lzero}
	that $G_{\lowbar}L(0)=\Id$. Thus
	$G_{\phi}L(0)=(\phi\otimes \text{Id}_{L(0)})\restrict_0\neq 0$,
	so $G_{\phi}$ is non-zero and hence $G$ is faithful.
	
	Now, let $m^*\in M^*$, $m^*\neq 0$ be a lowest weight vector
	of weight $\mu\in\frak h^*$
	in the image of the map $\phi^*\colon N^*\rightarrow M^*$, and let
	$n^*\in N^*$ with $\phi^*(n^*)=m^*$. Let $\lambda\in\frak h^*$
	be the antidominant weight, i.e. with $L(\lambda)=\Delta(\lambda)\in\catO_0$,
	and consider
	$F_{\phi}\Delta(\lambda)\colon F_{N}\Delta(\lambda)\rightarrow F_M\Delta(\lambda)$.
	Let $v\in\Delta(\lambda)$ denote a non-zero highest weight vector of $\Delta(\lambda)$.
	
	Since $\mu$ is a lowest weight of $\phi^*(N^*)$ and $N\in\catO_0$,
	it follows that $\lambda+\mu\leqslant 0$ and $\Delta(\lambda+\mu)\in\catO_0$.
	In particular, by the proof of Proposition~\ref{prop:freeproj},
	both $n^*\otimes v$ and $m^*\otimes v$ represent non-zero
	elements $\overline{n^*\otimes v}$ and
	$\overline{m^*\otimes v}$ in
	\[
		(N^*\otimes \Delta(\lambda))^{\leqslant0}\big\restrict_0 = F_N\Delta(\lambda)
	\]
	and
	\[
		(M^*\otimes \Delta(\lambda))^{\leqslant0}\big\restrict_0=F_M\Delta(\lambda),
	\]
	respectively. In particular, since
	\[
		F_{\phi}\Delta(\lambda) =
		(\phi^*\otimes \Id_{\Delta(\lambda)})^{\leqslant 0}\big\restrict_0
	\]
	we see that
	\[
		\bigl(F_{\phi}\Delta(\lambda)\bigr)(\overline{n^*\otimes v})
		= \overline{\phi^*(n^*)\otimes v}
		= \overline{m^*\otimes v}\neq 0.
	\]
	Hence $F_{\phi}$ is non-zero, proving
	that $F$ is faithful. By duality, it follows that $H$ is faithful.
\end{proof}

We now conclude the proof of Theorem~\ref{thm:maintheorem} by
showing a slightly stronger statement than ``$X_M\cong X_N$
if and only if $M\cong N$''.

\begin{proposition}
	Let $M, N\in\catO_0$ with $M\ncong N$. Then
	\begin{enumerate}[(a)]
		\item $F_M\vert_\text{proj}\ncong F_N\vert_\text{proj}$,
		\item $G_M\vert_\text{tilt}\ncong G_N\vert_\text{tilt}$, and
		\item $H_M\vert_\text{inj}\ncong H_N\vert_\text{inj}$,
	\end{enumerate}
	where $\vert_\text{proj}$, $\vert_\text{tilt}$ and
	$\vert_\text{inj}$ denote the restrictions to the additive
	categories of projective,
	tilting and injective modules, respectively.
\end{proposition}

\begin{proof}
	We start by noting that if $G_M\cong G_N$, then
	\begin{equation}
		\label{eq:gmcgn}
		M \cong G_ML(0)\cong G_NL(0)\cong N.
	\end{equation}
	Assume that $F_M\vert_\text{proj}\cong F_N\vert_\text{proj}$.
	Since $F_M$ and $F_N$ are right exact, it follows
	by taking projective presentations that
	$F_MK\cong F_NK$ for any $K\in\catO_0$, i.e. $F_N\cong F_M$.
	By the uniqueness of right adjoints, this implies that
	$G_M\cong G_N$ so $M\cong N$ by \eqref{eq:gmcgn},
	and hence we have proved part (a). Part (c) follows from
	(a) by duality (as in the proof of Proposition~\ref{prop:acyclic}).
	
	For part (b), assume that $G_M\vert_\text{tilt}\cong G_N\vert_\text{tilt}$.
	We recall that each projective module $P\in\catO_0$
	has a tilting co-resolution, i.e. there are tilting modules
	$T_0, \dotsc, T_k\in\catO_0$ such that the sequence
	\[
		0\rightarrow P\rightarrow T_0\rightarrow\cdots\rightarrow T_k\rightarrow 0
	\]
	is exact (for details, see \cite[Lemma~6]{ringel}).
	Since $G_N$ and $G_M$ are exact and agree on the additive category of tilting
	modules, this induces the following commutative diagram with exact
	rows.
	\[
		{\setlength{\arraycolsep}{2pt}
		\begin{array}{ccccccccccc}
			0 & \rightarrow & G_MP & \rightarrow & G_MT_0 &
				\rightarrow & \cdots & \rightarrow & G_MT_k & \rightarrow & 0 \\
			& & & & \rotatebox[origin=c]{270}{$\cong$}
			& & & & \rotatebox[origin=c]{270}{$\cong$} \\
			0 & \rightarrow & G_NP & \rightarrow & G_NT_0 &
				\rightarrow & \cdots & \rightarrow & G_NT_k & \rightarrow & 0
		\end{array}}
	\]
	By the Five Lemma this induces an isomorphism
	$G_MP\cong G_NP$, which furthermore is natural, since all isomorphisms in the
	above diagram are natural. Hence $G_M$ and $G_N$ are naturally equivalent
	on projective modules,
	so by the right exactness $G_M\cong G_N$ as in the proof of part (a).
	By \eqref{eq:gmcgn} we have $M\cong N$, as required.
\end{proof}

\begin{proposition}
	\label{prop:mdnnprojtiltinj}
	For all $M\in\calF_0(\Delta)$, $N\in\calF_0(\nabla)$ we have that
	\begin{enumerate}[(a)]
		\item $F_NM$ is projective,
		\item $G_MN\cong G_NM$ is a tilting module, and
		\item $H_MN$ is injective.
	\end{enumerate}
\end{proposition}

\begin{proof}
	For part (a), we need to show that $\Hom(F_NM, \lowbar)$ is exact. Since
	\[
		\Hom(F_NM, \lowbar) \cong \Hom(M, G_N\lowbar),
	\]
	it is equivalent to show that $\Hom(M, G_N\lowbar)$ is exact.
	By Proposition~\ref{prop:preservestilt}, $G_N\lowbar$ maps any module
	to a module with a dual Verma flag, since $N\in\calF(\nabla)$.
	Hence, as $G_N\lowbar$ is exact,
	it maps an exact sequence to an exact sequence of modules in $\calF(\nabla)$.
	Finally, $\Hom(M, \lowbar)$ is acyclic on $\calF(\nabla)$
	since $M\in\calF(\Delta)$ (see \cite[Corollary~2]{ringel}), so
	applying $\Hom(M, \lowbar)$ to an exact sequence of modules in
	$\calF(\nabla)$ again yields an
	exact sequence, i.e. $\Hom(M, G_N\lowbar)$ is exact.	

	Part (c) follows from (a) by duality. Finally, part (b) follows
	directly from Proposition~\ref{prop:preservestilt}.
\end{proof}

\begin{corollary}
	For all $M\in\calT_0$, $F_M$ maps tilting modules to projective modules,
	and $H_M$ maps tilting modules to injective modules.
\end{corollary}

In general it is quite difficult to compute $F_MN$ and $H_MN$, but the following
is a nice special case.

\begin{proposition}
	\label{prop:fnd}
	For each $\lambda\in\frak h^*$ with $\Delta(\lambda)\in\catO_0$
	we have
	\begin{align*}
		F_{\nabla(\lambda)}\Delta(\lambda) &\cong \Delta(0), \text{ and} \\
		H_{\Delta(\lambda)}\nabla(\lambda) &\cong \nabla(0).
	\end{align*}
\end{proposition}

\begin{proof}
	Let $\mu\in\frak h^*$ be such that $\mu<0$ and $L(\mu)\in\catO_0$.
	Since $\mu<0$ it follows that
	$\bigl(G_{\nabla(\lambda)}L(\mu)\bigr)_\lambda=\set{0}$, so
	\[
		\dim\Hom\bigl(F_{\nabla(\lambda)}\Delta(\lambda), L(\mu)\bigr) \cong
		\dim\Hom\bigl(\Delta(\lambda), G_{\nabla(\lambda)}L(\mu)) = 0.
	\]
	On the other hand, we have
	\begin{align*}
		\dim\Hom\bigl(F_{\nabla(\lambda)}\Delta(\lambda), L(0)\bigr)
		&\cong \dim\Hom\bigl(\Delta(\lambda), G_{\nabla(\lambda)}L(0)\bigr) \\
		&\cong \dim\Hom\bigl(\Delta(\lambda), \nabla(\lambda)\bigr) \\
		&= 1,
	\end{align*}
	so $F_{\nabla(\lambda)}\Delta(\lambda)$ has simple top $L(0)$.
	By Proposition~\ref{prop:mdnnprojtiltinj}
	$F_{\nabla(\lambda)}\Delta(\lambda)$ is projective, and hence
	\[
		F_{\nabla(\lambda)}\Delta(\lambda)\cong\Delta(0).
	\]
	The second statement follows by duality.
\end{proof}

\begin{proposition}
	\label{prop:nattrans}
	There are natural transformations
	\begin{enumerate}[(a)]
		\item $G_{\Delta(0)}\epi \Id$, $\Id\hookrightarrow G_{\nabla(0)}$,
		\item $\Id\hookrightarrow H_{\Delta(0)}$, $F_{\nabla(0)}\epi\Id$.
	\end{enumerate}
\end{proposition}

\begin{proof}
	Since $F_{L(0)}\cong G_{L(0)}\cong H_{L(0)}\cong\Id$, together
	with the fact that $F$ is right exact, $G$ is exact
	and $H$ is left exact,
	this follows by applying the functors $F$, $G$ and $H$ to
	the canonical homomorphisms $\Delta(0)\epi L(0)$ and
	$L(0)\hookrightarrow\nabla(0)$.
\end{proof}

\section{(Co-)Monad structures}
\label{sec:comonad}

We briefly recall the definition of a monad and a comonad (sometimes called
triple and cotriple, respectively), for details see~\cite{maclane, weibel}.
A monad $(\mho, \nabla, \eta)$ on a category $\calC$ is
an endofunctor $\mho\colon\calC\rightarrow\calC$ together with
two natural transformations $\nabla\colon \mho^2\rightarrow \mho$ and
$\eta\colon \Id\rightarrow \mho$ such that the diagrams
\newsavebox{\midalignbox}%
\newcommand{\midalign}[1]{\savebox{\midalignbox}{#1}%
	\raisebox{-0.5\ht\midalignbox}{\usebox{\midalignbox}}}%
\begin{equation}
	\label{eq:monaddiagrams}
	\midalign{\includegraphics{monad.1}}
	\qquad\text{and}\qquad
	\midalign{\includegraphics{monad.2}}
\end{equation}
commute. Dually, a comonad $(\Omega, \Delta, \varepsilon)$ on a category $\calC$
is an endofunctor $\Omega\colon\calC\rightarrow\calC$ together with
two natural transformations $\Delta\colon \Omega\rightarrow \Omega^2$ and
$\varepsilon\colon \Omega\rightarrow \Id$ such that the diagrams
\begin{equation}
	\label{eq:comonaddiagrams}
	\midalign{\includegraphics{monad.3}}
	\qquad\text{and}\qquad
	\midalign{\includegraphics{monad.4}}
\end{equation}
commute.

Fix a non-zero highest weight vector $v$ of $\Delta(0)$. Recall that $\enva{\frak g}$
admits a coalgebra structure with counit
$\tilde\epsilon\colon\enva{\frak g}\rightarrow\bbC$ and comultiplication
$\tilde\Delta\colon\enva{\frak g}\rightarrow\enva{\frak g}\otimes\enva{\frak g}$.
This induces two
homomorphisms
\begin{align}
	\label{eq:comone}
	D&\colon\Delta(0)\hookrightarrow\Delta(0)\otimes\Delta(0),
		uv\mapsto \tilde\Delta(u)(v\otimes v), \\
	\label{eq:comtwo}
	E&\colon\Delta(0)\epi L(0),
		uv\mapsto \tilde\epsilon(u),
\end{align}
for $u\in\enva{\frak n_-}$, where we identify $L(0)$ with $\bbC$
via $\overline v\mapsto 1$.

\begin{proposition}
	\label{prop:comonad}
	The homomorphisms~\eqref{eq:comone} and~\eqref{eq:comtwo}
	induce a co\-mon\-ad $(\Delta(0)\otimes\lowbar, \Delta, \epsilon)$ on $\tcatO$
	with $\Delta$ injective and $\epsilon$ surjective, and dually
	a monad $(\nabla(0)\otimes\lowbar, \nabla, \eta)$ with $\nabla$ surjective
	and $\eta$ injective.
\end{proposition}

\begin{proof}
	Fix $M\in\tcatO$. Applying the functor $\lowbar\otimes M$
	to~\eqref{eq:comone} and~\eqref{eq:comtwo} we obtain the homomorphisms
	(where as above we identify $L(0)$ with $\bbC$)
	\[
		\Delta_M\defeq D\otimes \text{Id}_M
			\colon \Delta(0)\otimes M\hookrightarrow\Delta(0)\otimes\Delta(0)\otimes M,
	\]
	and
	\[
		\epsilon_M\defeq E\otimes \text{Id}_M\colon \Delta(0)\otimes M\epi M.
	\]
	
	By the proof of Proposition~\ref{prop:preservestilt},
	$\Delta(0)\otimes M$ is generated by elements
	of the form $v\otimes m$, $m\in M$. For such an element,
	it is trivial to show that
	\[
		\bigl((\epsilon_{\Delta(0)\otimes M})\circ\Delta_M\bigr)(v\otimes m)
		= \bigl((\text{Id}_{\Delta(0)}\otimes \epsilon_M)\circ\Delta_M\bigr)(v\otimes m)
		= v\otimes m,
	\]
	and
	\begin{align*}
		\bigl((\Delta_{\Delta(0)\otimes M})\circ\Delta_M\bigr)(v\otimes m)
		&= \bigl((\text{Id}_{\Delta(0)}\otimes \Delta_M)\circ\Delta_M\bigr)(v\otimes m) \\
		&= v\otimes v\otimes v\otimes m,
	\end{align*}
	so the diagrams~\eqref{eq:comonaddiagrams} commute, proving
	that $(\Delta(0)\otimes\lowbar, \Delta, \epsilon)$ is a comonad on $\tcatO$.
	
	Applying $\star\circ(\lowbar\otimes M^\star)$ to~\eqref{eq:comone}
	and~\eqref{eq:comtwo} gives the homomorphisms
	\begin{align*}
		\nabla_M&\defeq (\Delta_{M^\star})^\star\colon
			\nabla(0)\otimes\nabla(0)\otimes M\epi\nabla(0)\otimes M, \text{ and}\\
		\eta_M&\defeq (\epsilon_{M^\star})^\star\colon
			M\hookrightarrow \nabla(0)\otimes M.
	\end{align*}
	By duality, the diagrams~\eqref{eq:monaddiagrams} commute.
\end{proof}

We can refine this result to the category $\catO_0$.

\begin{theorem}
	\label{thm:comonado0}
	The homomorphisms~\eqref{eq:comone} and~\eqref{eq:comtwo}
	induce a comonad $(G_{\Delta(0)}, \bar\Delta, \bar\epsilon)$ on
	$\catO_0$, with $\bar\Delta$ injective and $\bar\epsilon$ surjective,
	and dually a monad $(G_{\nabla(0)}, \bar\nabla, \bar\eta)$ on
	$\catO_0$, with $\bar\nabla$ surjective and $\bar\eta$ injective.
\end{theorem}

We prove Theorem~\ref{thm:comonado0} in parts, throughout the rest
of this section. Define
\begin{align*}
	\bar\Delta_M&\colon G_{\Delta(0)}M\rightarrow G_{\Delta(0)}G_{\Delta(0)} M,\text{ and} \\
	\bar\epsilon_M&\colon G_{\Delta(0)}M\rightarrow M,
\end{align*}
by
\begin{align*}
	\bar\Delta_M&\defeq
		\pi_{G_{\Delta(0)}G_{\Delta(0)}M}\circ\Delta_M\circ\iota_{G_{\Delta(0)}M},
		\text{ and} \\
	\bar\epsilon_M&\defeq \epsilon_M\circ\iota_{G_{\Delta(0)}M},
\end{align*}
where $\pi_x$ and $\iota_x$ as before denotes natural projections and injections.
Let $\bar\Delta$ and $\bar\epsilon$ be the natural transformations
corresponding to $\bar\Delta_M$ and $\bar\epsilon_M$.

\begin{remark}
	\label{rem:comodfactors}
	Similarly as in the case of Remark~\ref{rem:ntransfactors}, by central character
	considerations we see that $\pi_{G_{\Delta(0)}G_{\Delta(0)}}\circ\Delta_M$
	factors through $\bar\Delta_M$, and $\epsilon_M$ factors through
	$\bar\epsilon_M$, i.e. the diagrams
	\[
		\includegraphics{comodfactors.1}\qquad
		\includegraphics{comodfactors.2}
	\]
	commute.
\end{remark}

\begin{lemma}
	\label{lem:lcomm}
	The left of the diagrams~\eqref{eq:comonaddiagrams}
	for the triple $(G_{\Delta(0)}, \bar\Delta, \bar\epsilon)$ commutes.
\end{lemma}

\begin{proof}
	Fix $M\in\catO_0$, with a weight basis $m_1, m_2, \dots\in M$,
	and consider an element
	\[
		\sum_{i=1}^k(u_iv)\otimes m_i\in G_{\Delta(0)}M,
	\]
	where $u_1$, $\dots$, $u_k\in\enva{\frak n_-}$. Applying
	$\Delta_M$ yields, after collecting the elements of the form
	$v\otimes \lowbar\otimes \lowbar$,
	\begin{equation}
		\label{eq:blahhh}
		\sum_{i=1}^kv\otimes (u_iv)\otimes m_i +
			\sum_{i, j}(u'_{ij}v)\otimes (u''_{ij}v)\otimes m_i,
	\end{equation}
	where $\tilde\epsilon(u'_{ij})=0$ for all $u'_{ij}$ in the sum on the
	right. Hence, when applying
	\[
		(\epsilon\otimes\text{Id}_{G_{\Delta(0)}M})\circ
			(\text{Id}_{\Delta(0)}\otimes \pi_{G_{\Delta(0)}M}),
	\]
	the right hand sum of~\eqref{eq:blahhh} maps to zero, while
	the left hand sum of~\eqref{eq:blahhh} maps to
	\[
		\sum_{i=1}^k (u_iv)\otimes m_i.
	\]
	Hence $\bar\epsilon_{G_{\Delta(0)}M}\circ\bar\Delta_M=\text{Id}_{G_{\Delta(0)}M}$,
	so the upper triangle of the left diagram of~\eqref{eq:comonaddiagrams}
	commutes.
	
	For the lower triangle, consider the following diagram.
	\[
		\includegraphics{monad_o0.3}
	\]
	The left square and the triangle commutes by Remark~\ref{rem:comodfactors},
	and the right quadrangle commutes by Remark~\ref{rem:ntransfactors}, and
	hence the diagram commutes.
	By Proposition~\ref{prop:comonad}, the top row equals
	$\text{Id}_{\Delta(0)\otimes M}$, and hence the bottom row
	equals $\text{Id}_{G_{\Delta(0)}M}$, as required.
\end{proof}

\begin{corollary}
	\label{cor:injsur}
	The homomorphism $\bar\epsilon_M$ is surjective and
	the homomorphism $\bar\Delta_M$ is injective.
\end{corollary}

\begin{proof}
	Since $\bar\epsilon_M=\epsilon_M\circ\iota_{G_{\Delta(0)}M}$ it follows
	that $\bar\epsilon_M$ is surjective as $\epsilon_M$ is surjective.
	By Lemma~\ref{lem:lcomm} we have
	$G_{\Delta(0)}\bar\epsilon_M\circ\bar\Delta_M=\text{Id}_{G_{\Delta(0)}M}$,
	so $\bar\Delta_M$ is injective since $\text{Id}_{G_{\Delta(0)}M}$ is
	injective.
\end{proof}

\begin{lemma}
	\label{lem:rcomm}
	The right of the diagrams~\eqref{eq:comonaddiagrams}
	for the triple $(G_{\Delta(0)}, \bar\Delta, \bar\epsilon)$ commutes.
\end{lemma}

\begin{proof}
	We claim that the diagrams
	\[
		\includegraphics[width=0.8\textwidth]{monad_o0.1}
	\]
	and
	\[
		\includegraphics[width=0.8\textwidth]{monad_o0.2}
	\]
	commute. For the first diagram, the left and top right squares
	commute by Remark~\ref{rem:comodfactors}, and the bottom right square
	commutes by Remark~\ref{rem:ntransfactors}. For the second diagram,
	the left and bottom right squares commute by Remark~\ref{rem:comodfactors}.
	For the top right square, we note that
	\begin{align*}
		\Delta_{\Delta(0)\otimes M} &= D\otimes \text{Id}_{\Delta(0)\otimes M}, \text{ and} \\
		\Delta_{G_{\Delta(0)}M} &= D\otimes \text{Id}_{G_{\Delta(0)}M},
	\end{align*}
	so the square commutes, since
	\begin{align*}
		D\otimes \pi_{G_{\Delta(0)}M}
		&= (\text{Id}_{\Delta(0)\otimes \Delta(0)}\otimes \pi_{G_{\Delta(0)}M})
			\circ (D\otimes \text{Id}_{\Delta(0)\otimes M}) \\
		&= (D\otimes \text{Id}_{G_{\Delta(0)}M})
			\circ (\text{Id}_{\Delta(0)}\otimes \pi_{G_{\Delta(0)}M}).
	\end{align*}
	Thus both diagrams commute. Hence, since
	\[
		(\text{Id}_{\Delta(0)}\otimes\Delta_M)\circ\Delta_M =
		\Delta_{\Delta(0)\otimes M}\circ\Delta_M
	\]
	by Proposition~\ref{prop:comonad}, and the fact that projections commute,
	it follows that
	\[
		G_{\Delta(0)}\bar\Delta_M\circ \bar\Delta_M
		= \bar\Delta_{G_{\Delta(0)}M}\circ\bar\Delta_M,
	\]
	and thus the right of the diagrams~\eqref{eq:comonaddiagrams}
	commute.
\end{proof}

From Lemma~\ref{lem:lcomm} and Lemma~\ref{lem:rcomm} it follows that
$(G_{\Delta(0)}, \bar\Delta, \bar\epsilon)$ is a comonad on
$\catO_0$, and $\bar\Delta$ is
injective and $\bar\epsilon$ is surjective by Corollary~\ref{cor:injsur}.
Finally, as in the proof of Proposition~\ref{prop:comonad},
setting
\begin{align*}
	\bar\nabla_M&\defeq (\bar\Delta_{M^\star})^\star, \text{ and}\\
	\bar\eta_M&\defeq (\bar\epsilon_{M^\star})^\star,
\end{align*}
gives a monad $(G_{\nabla(0)}, \bar\nabla, \bar\eta)$ with $\bar\nabla$
surjective and $\bar\eta$ injective, by duality, which concludes
the proof of Theorem~\ref{thm:comonado0}.

\section{Parabolic subcategories}
\label{sec:parabolic}

All the previous results can be generalized to the case of the parabolic
analogue of $\catO$, in the sense of Rocha-Caridi (see for example~\cite{rc, irving}).
Let $\frak p\subseteq \frak b$ be a parabolic subalgebra of $\frak g$,
let $\frak m\subseteq \frak n_-$ with
\[
	\frak g = \frak m\oplus \frak p,
\]
and let $R_{\frak m}$ be the roots of $\frak m$. The parabolic
analogies of $\catO$, $\tcatO$, $\calF(\Delta)$, etc.~are obtained by
substituting $\frak n_-$ by $\frak m$, $\frak b$ by $\frak p$,
and $R_-$ by $R_{\frak m}$, in the corresponding definition.
Thus, for example, $\catO^{\frak p}$ is defined
as the full subcategory of the category of $\frak g$-modules
consisting of weight modules that are finitely generated
as $\enva{\frak m}$-modules, and $\calF^{\frak p}(\Delta)$ is
the full subcategory of $\catO^{\frak p}$ that are free as $\enva{\frak m}$-modules.
Similarly, the partial order $\leqslant$ on
$\frak h^*$ is replaced by $\leqslant_{\frak p}$ defined as
$\lambda\leqslant_{\frak p}\mu$ if and only if
$\lambda-\mu\in\bbN_0R_{\frak m}$, and so on.

Recall that a generalised Verma module in $\catO^{\frak p}$
is an element of $\calF^{\frak p}(\Delta)$
that is generated by a highest weight vector (for details,
see \cite{lepowsky}). We denote the generalised Verma module
generated by a highest weight vector of weight $\lambda\in\frak h^*$ by
$\Delta^{\frak p}(\lambda)$. Furthermore, the objects
in $\calF^{\frak p}(\Delta)$ are precisely the objects in $\catO^{\frak p}$
that have a generalised Verma filtration.

Almost all statements and proofs of the previous sections hold
verbatim with these substitutions. The exception is
Proposition~\ref{prop:freeproj}, which needs to be restated in the
following (rather complicated) way.
Let $\frak g^{\frak p}$ denote the semisimple part of $\frak p$.

\begin{proposition}
	\label{prop:pfreeproj}
	Let $M$ be a $\enva{\frak m}$-free module with a $\enva{\frak m}$-basis
	\[
		\bigl\{\,v_{ij}\,\big\vert\,i\in I, 1\leq j\leq k_i\,\bigr\}
	\]
	for some index set $I$ and non-negative integers $k_i$ such that
	\[
		L_i \defeq \enva{\frak g^{\frak p}}\{\,v_{ij}\,\vert\,1\leq j\leq k_i\,\}
	\]
	is a $k_i$-dimensional $\frak g^{\frak p}$-module with basis
	$v_{i1}, \dotsc, v_{ik_i}$. Then
	\[
		M^{\leqslant\lambda} = M^{\leqslant_{\frak p}\lambda}
		\cong
		\bigoplus_{\substack{i\in I,\\ L_i\leqslant\lambda}}
		\enva{\frak n_-}\{v_{i1}, \dotsc, v_{ik_i}\},
	\]
	where $L_i\leqslant\lambda$ if $\weight(v_{ij})\leqslant\lambda$
	for all $1\leq j\leq k_i$.
\end{proposition}

\begin{proof}
	By completely analogous arguments as in the proof of Proposition~\ref{prop:freeproj},
	it follows that
	\[
		M^{\nleqslant\lambda} =
		\sum_{\substack{i\in I,\\ L_i\nleqslant\lambda}}
		\enva{\frak n_-}\{v_{i1}, \dotsc, v_{ik_i}\},
	\]
	and hence the claim follows.
\end{proof}

All objects of $\calF^{\frak p}(\Delta)$ satisfy the requirements
of Proposition~\ref{prop:pfreeproj}, and a straightforward argument shows
that $M\otimes N^*$ does as well, for all $M\in\calF^{\frak p}(\Delta)$
and $N\in\catO^{\frak p}$. In particular, we conclude that the arguments used in
Sections~\ref{sec:main} and~\ref{sec:comonad} all translate to the parabolic
setting.

The main results for the category $\catO_0^{\frak p}$ are thus the following.

\begin{theorem}
	There exist faithful functors
	\begin{align*}
		F^{\frak p}&\colon\catO_0^{\frak p}\hookrightarrow\PFun(\catO_0^{\frak p})^{\text{op}},
			M\mapsto F^{\frak p}_M,\\
		G^{\frak p}&\colon\catO_0^{\frak p}\hookrightarrow\TFun(\catO_0^{\frak p}),
			M\mapsto G^{\frak p}_M,\\
		H^{\frak p}&\colon\catO_0^{\frak p}\hookrightarrow\IFun(\catO_0^{\frak p})^{\text{op}},
			M\mapsto H^{\frak p}_M,
	\end{align*}
	all three satisfying $X_M\cong X_N$ if and only if $M\cong N$
	(where $X=F^{\frak p}, G^{\frak p}, H^{\frak p}$).
\end{theorem}

\begin{proof}
	These are just the restrictions of $F$, $G$, and $H$ to
	$\catO_0^{\frak p}$.
\end{proof}

\begin{proposition}
\label{prop:pacyclic}
	For any $M\in\catO_0^{\frak p}$ the following holds:
	\begin{enumerate}[(a)]
		\item $F^{\frak p}_M$ and $G^{\frak p}_M$ preserve $\calF_0^{\frak p}(\Delta)$ and are
			acyclic on it.
		\item $G^{\frak p}_M$ and $H^{\frak p}_M$ preserve $\calF_0^{\frak p}(\nabla)$ and are
			acyclic on it.
	\end{enumerate}
\end{proposition}

\begin{proposition}
	\label{prop:pmdnnprojtiltinj}
	For all $M\in\calF_0^{\frak p}(\Delta)$, $N\in\calF_0^{\frak p}(\nabla)$ we have that
	\begin{enumerate}[(a)]
		\item $F^{\frak p}_NM$ is projective,
		\item $G^{\frak p}_MN\cong G^{\frak p}_NM$ is a tilting module, and
		\item $H^{\frak p}_MN$ is injective.
	\end{enumerate}
\end{proposition}

\begin{corollary}
	For all $M\in\calT_0^{\frak p}$, $F^{\frak p}_M$ maps tilting modules to projective modules,
	and $H^{\frak p}_M$ maps tilting modules to injective modules.
\end{corollary}

\begin{proposition}
	\label{prop:pfnd}
	For each $\lambda\in\frak h^*$ with $\Delta^{\frak p}(\lambda)\in\catO_0^{\frak p}$
	we have
	\begin{align*}
		F^{\frak p}_{\nabla^{\frak p}(\lambda)}\Delta^{\frak p}(\lambda) &\cong \Delta^{\frak p}(0),
			\text{ and} \\
		H^{\frak p}_{\Delta^{\frak p}(\lambda)}\nabla^{\frak p}(\lambda) &\cong \nabla^{\frak p}(0).
	\end{align*}
\end{proposition}

\begin{proposition}
	The canonical homomorphisms $\Delta^{\frak p}(0)\epi L(0)$
	and $\Delta^{\frak p}(0)\hookrightarrow\Delta^{\frak p}(0)\otimes\Delta^{\frak p}(0)$
	induce a comonad $(G_{\Delta^{\frak p}(0)}, \Delta^{\frak p}, \epsilon^{\frak p})$
	on $\catO_0^{\frak p}$
	with $\Delta^{\frak p}$ injective and $\epsilon^{\frak p}$ surjective, and dually
	a monad $(G_{\nabla^{\frak p}(0)}, \nabla^{\frak p}, \eta^{\frak p})$
	with $\nabla^{\frak p}$ surjective and $\eta^{\frak p}$ injective.
\end{proposition}

\section{An example: $\frak{sl}_3(\bbC)$}
\label{sec:example}

In conclusion we will compute the `multiplication table'
given by $G_MN$ and $F_MN$, where $M$ and $N$ run through
the simple modules of $\catO_0$ for the algebra $\frak g=\frak{sl}_3(\bbC)$,
see Tables~\ref{tab:gmult} and~\ref{tab:fmult}.
Let $\alpha, \beta\in\frak h^*$ denote the simple roots, let
$s$ and $t$ be the corresponding simple reflections (i.e. with $s(\alpha)=-\alpha$
and $t(\beta)=-\beta$), and fix
a Weyl-Chevalley basis $X_{\pm\alpha}$, $X_{\pm\beta}$,
$X_{\pm(\alpha+\beta)}$, $H_\alpha$, $H_\beta$.

The `dot' action of the Weyl group
$W=S_3$ on $\frak h^*$ is defined by
\[
	w\cdot\lambda\defeq w(\lambda+\rho)-\rho
\]
for an element $w\in W$, where $\rho\in\frak h^*$ is half the sum
of the positive roots.
We set $L(w)\defeq L(w\cdot 0)$ for $w\in W$. Let $e$ denote the identity in $W$.
There are two proper parabolic subalgebras,
$\frak p^{\alpha}\defeq\frak b+\langle X_{-\alpha}\rangle_\bbC$
and $\frak p^{\beta}\defeq\frak b+\langle X_{-\beta}\rangle_\bbC$.

\begin{figure}
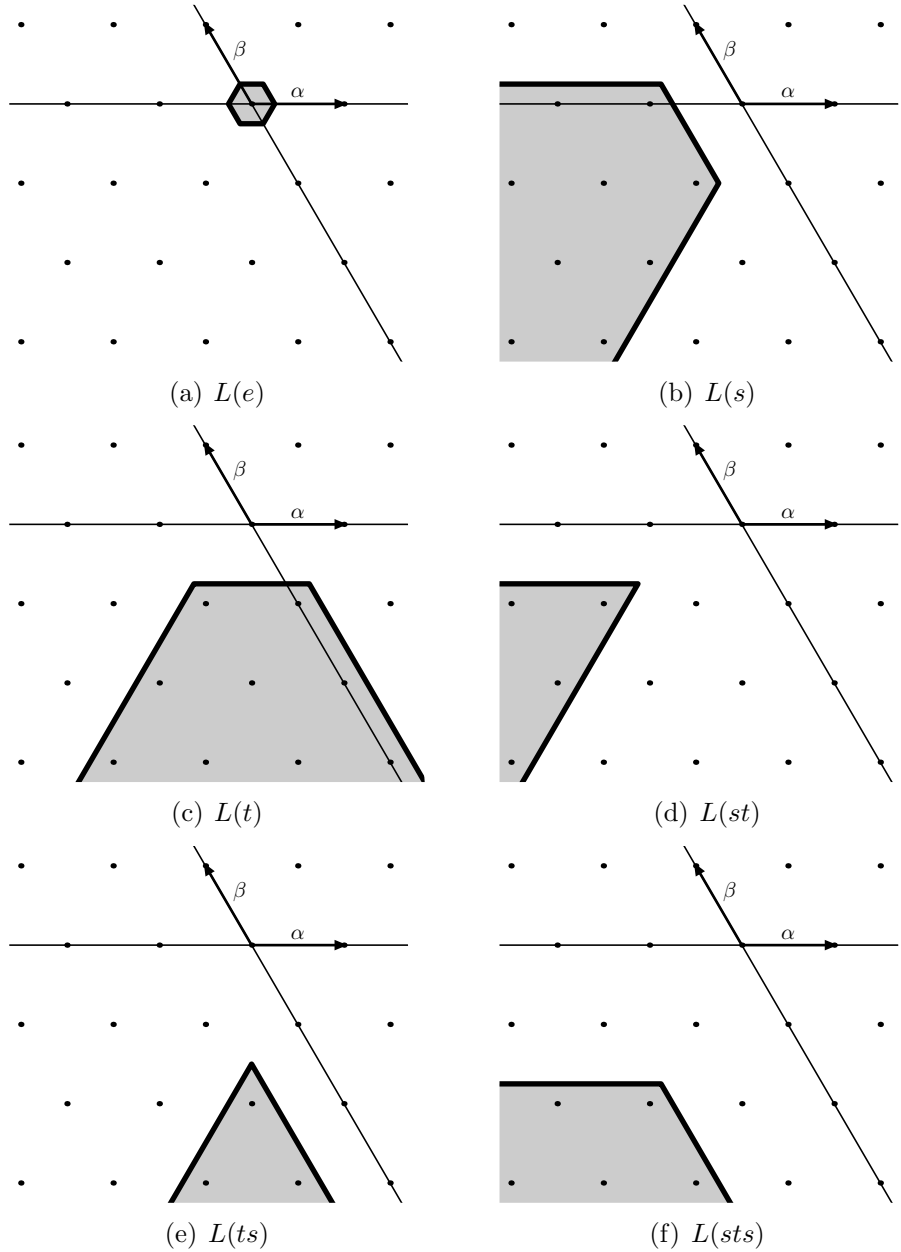

\begin{center}
	\mbox{
		\subfigure[$L(e)$]{\includegraphics[width=0.4\textwidth]{sl3o0enkla.1}} \qquad
		\subfigure[$L(s)$]{\includegraphics[width=0.4\textwidth]{sl3o0enkla.2}}
		}
	\mbox{
		\subfigure[$L(t)$]{\includegraphics[width=0.4\textwidth]{sl3o0enkla.3}} \qquad
		\subfigure[$L(st)$]{\includegraphics[width=0.4\textwidth]{sl3o0enkla.4}}
		}
	\mbox{
		\subfigure[$L(ts)$]{\includegraphics[width=0.4\textwidth]{sl3o0enkla.5}} \qquad
		\subfigure[$L(sts)$]{\includegraphics[width=0.4\textwidth]{sl3o0enkla.6}}
		}
	\caption{The simple modules in $\catO_0$ for the algebra
		$\frak{sl}_3(\bbC)$. Each dot is an integral weight, and the
		grey areas show the support of the corresponding module. Each non-empty
		weight space has dimension $1$ except for $L(sts)$, for which the dimensions
		are given by Kostant's function.}
	\label{fig:sl3simple}
\end{center}
\end{figure}

\newcommand{\smd}[4]{\phantom{\scriptstyle#2\,}%
	\hbox{\vtop{\vbox{\hbox{\clap{$\scriptstyle #1$}}\nointerlineskip\smallskip%
	\hbox{\clap{$\scriptstyle #2\;#3$}}}\nointerlineskip\smallskip%
	\hbox{\clap{$\scriptstyle#4$}}}}\phantom{\,\scriptstyle#3}}

\begin{table}
\begin{center}
	\begin{tabular}{|c||c|c|c|c|c|c|}
		\hline
		\vphantom{$\Big)$}$G_MN$ & $L(e)$ & $L(s)$ & $L(t)$ & $L(st)$ & $L(ts)$ & $L(sts)$ \\
		\hline\hline
		\vphantom{$\Big)$}$L(e)$ & $L(e)$ & $L(s)$ & $L(t)$ & $L(st)$ & $L(ts)$ & $L(sts)$ \\
		\hline
		$L(s)$ & $L(s)$ & $L(st)$ &
			$\smd{L(sts)}{L(st)}{L(ts)}{L(sts)}$ &  $0$   & $L(sts)$ &  $0$   \\
		\hline
		$L(t)$ & $L(t)$ & $\smd{L(sts)}{L(st)}{L(ts)}{L(sts)}$
			& $L(ts)$ & $L(sts)$ &  $0$   & $0$    \\
		\hline
		\vphantom{$\Big)$}$L(st)$ & $L(st)$ &  $0$   & $L(sts)$ &  $0$   & $0$    & $0$ \\
		\hline
		\vphantom{$\Big)$}$L(ts)$ & $L(ts)$ & $L(sts)$ &   $0$  &  $0$   & $0$    & $0$ \\
		\hline
		\vphantom{$\Big)$}$L(sts)$ & $L(sts)$ &  $0$   &   $0$  &  $0$   & $0$    & $0$ \\
		\hline
	\end{tabular}
	\caption{The ``multiplication table'' for the bifunctor $G$ on the
		simple modules in $\catO_0$ for $\frak{sl}_3(\bbC)$.}
	\label{tab:gmult}
\end{center}
\end{table}

\newcommand{\sms}[2]{\genfrac{}{}{0pt}{1}{#1}{#2}}

\begin{table}
\begin{center}
	\begin{tabular}{|c||c|c|c|c|c|c|}
		\hline
		\vphantom{$\Big)$}$F_MN$ & $L(e)$ & $L(s)$ & $L(t)$ & $L(st)$ & $L(ts)$ & $L(sts)$ \\
		\hline\hline
		\vphantom{$\Big)$} $L(e)$ & $L(e)$ & $L(s)$ & $L(t)$
			& $L(st)$ & $L(ts)$ & $\Delta(sts)$  \\
		\hline
		\vphantom{$\Big)$}$L(s)$ & $0$ & $L(e)$ & $0$ & $\Delta^{\frak p^\beta}(s)$ & $0$ & %
		 		$\Delta(ts)\oplus P(t)$\\
		\hline
		\vphantom{$\Big)$}$L(t)$ & $0$ & $0$ & $L(e)$ & $0$ & $\Delta^{\frak p^\alpha}(t)$ & %
				$\Delta(st)\oplus P(s)$\\
		\hline
		\vphantom{$\Big)$}$L(st)$ & $0$ & $0$ & $0$ &
				$\Delta^{\frak p^\beta}(e)$ & $0$ & $\Delta(t)$ \\
		\hline
		\vphantom{$\Big)$}$L(ts)$ & $0$ & $0$ & $0$ &
			$0$ & $\Delta^{\frak p^\alpha}(e)$ & $\Delta(s)$ \\
		\hline
		\vphantom{$\Big)$}$L(sts)$ & $0$ & $0$ & $0$ & $0$ & $0$ & $\Delta(e)$ \\
		\hline
	\end{tabular}
	\caption{The ``multiplication table'' for the bifunctor $F$ on the
		simple modules in $\catO_0$ for $\frak{sl}_3(\bbC)$.}
	\label{tab:fmult}
\end{center}
\end{table}

The first row and column for the $G$-table follow from
Remark~\ref{rem:lzero}. The zero entries are obtained by weight
arguments (e.g.\ \eqref{eq:tensorsupp} and \eqref{eq:dimotimes}).
Similarly one finds that $L(s)\otimes L(s)$ has a higest weight vector
of weight $st\cdot 0$. Since $L(s)$ is not $\enva{\langle X_{-\beta}\rangle}$-free,
it follows that $G_{L(s)}L(s)\cong L(st)$. By symmetry, $G_{L(t)}L(t)=L(ts)$.
Finally, for $G_{L(s)}L(t)\cong G_{L(t)}L(s)$, counting dimensions of the weight
spaces shows that $L(st)$ and $L(ts)$ each occur once in the Jordan-H\"older
decomposition, and $L(sts)$ occurs twice. Furthermore, since
$L(s)$ is $\enva{\langle X_{-\alpha}\rangle}$-free and $L(t)$ is
$\enva{\langle X_{-\beta}\rangle}$-free,
it follows that $G_{L(s)}L(t)$ is both
$\enva{\langle X_{-\alpha}\rangle}$-free and $\enva{\langle X_{-\beta}\rangle}$-free. Hence
neither $L(st)$ nor $L(ts)$ can occur in the socle of
$G_{L(s)}L(t)$. Finally, we have
\[
	\biggl(G_{L(s)}L(t)\biggr)^\star = G_{L(s)^\star}L(t)^\star = G_{L(s)}L(t),
\]
i.e. $G_{L(s)}L(t)$ is self-dual, so neither $L(st)$ nor $L(ts)$ can
occur in the top of $G_{L(s)}L(t)$. We conclude that the Loewy series
of $G_{L(s)}L(t)$ is
\[
	G_{L(s)}L(t)\cong
	\smd{L(sts)}{L(st)}{L(ts)}{L(sts)}.
\]

The corresponding table for $F$ is given in Table~\ref{tab:fmult}.
Since $F_{L(0)}M=M$, the first row is immediate. Furthermore,
by Proposition~\ref{prop:fnd}, and the fact that
$L(sts)=\Delta(sts)=\nabla(sts)$ we have $F_{L(sts)}L(sts)=\Delta(e)$.
Similarly, by Proposition~\ref{prop:pfnd} and the fact that
$L(st)=\Delta^{\frak p^\beta}(st)=\nabla^{\frak p^\beta}(st)$ we have
$F_{L(st)}L(st)=\Delta^{\frak p^\beta}(e)$ (and similarly for $F_{L(ts)}L(ts)$).
Using the
adjointness of $F$ and $G$, we can easily determine the
top of $F_{L(i)}L(j)$, i.e. $L(k)$ is in the top of $F_{L(i)}L(j)$ if and
only if $L(j)$ is in the socle of $G_{L(i)}L(k)$. In particular, this fact
and the $G$-table gives us all the $0$'s in the table.

The remaining cases need some additional case by case arguments. We begin with
$F_{L(s)}L(s)$. By adjointness, Table~\ref{tab:gmult} shows that
$F_{L(s)}L(s)$ has a simple top $L(e)$. Since $L(s)\in\catO_0^{\beta}$,
it follows that the possible modules are $L(e)$ and
\[
	\Delta^{\frak p^\beta}(e)=\sms{L(e)}{L(s)}.
\]
But by Proposition~\ref{prop:pacyclic} we have
$G_{L(s)}\Delta^{\frak p^\beta}(e)\in\calF_0^\beta(\Delta)$,
so by analysing the weights we see that
\[
	G_{L(s)}\Delta^{\frak p^\beta}(e)=\Delta^{\frak p^\beta}(s) = \sms{L(s)}{L(st)}.
\]
Hence
\begin{align*}
	\dim\Hom_{\frak g}\bigl(F_{L(s)}L(s), \Delta^{\frak p^\beta}(e)\bigr)
	&= \dim\Hom_{\frak g}\bigl(L(s), G_{L(s)}\Delta^{\frak p^\beta}(e)\bigr) \\
	&= \dim\Hom_{\frak g}\bigl(L(s), \Delta^{\frak p^\beta}(s)\bigr) \\
	&= 0,
\end{align*}
and so $F_{L(s)}L(s)\neq \Delta^{\frak p^\beta}(e)$ and we conclude that $F_{L(s)}L(s)=L(e)$. Analogously,
we get $F_{L(t)}L(t)=L(e)$.

Since $L(sts)=\Delta(sts)$, we have  $F_{L(st)}L(sts)\in\calF_0(\Delta)$
by Proposition~\ref{prop:acyclic}, and by the proof of
Proposition~\ref{prop:fpreservesdelta} we know that the Verma modules
$\Delta(\lambda)$ occuring in the Verma flag of $F_{L(st)}L(sts)$
are the ones satisfying $\lambda\in sts\cdot0 - \Supp L(st)$ and
$\lambda\leqslant 0$, with multiplicity equal to the dimension
of the weight space of $L(st)$ of weight $sts\cdot 0 - \lambda$.
The only such weight is $t\cdot 0$, with multiplicity $1$. Hence,
$F_{L(st)}L(sts)=\Delta(t)$. Analogously, $F_{L(ts)}L(sts)=\Delta(s)$.

Since $L(st)=\Delta^{\frak p^\beta}(st)$, we have $F_{L(s)}L(st)\in\calF_0^\beta(\Delta)$.
By a similar analysis as for $F_{L(st)}L(sts)$, using the proof of
Proposition~\ref{prop:pfreeproj}, we find that $F_{L(s)}L(st)$ has only
one generalised Verma quotient, $\Delta^{\frak p^\beta}(s)$, so
$F_{L(s)}L(st)=\Delta^{\frak p^\beta}(s)$.
Similarly, $F_{L(t)}L(ts)=\Delta^{\frak p^\alpha}(t)$.

Finally, for $F_{L(s)}L(sts)$, by the same analysis as for
$F_{L(st)}L(sts)$ we have that $F_{L(s)}L(sts)$ has a Verma flag
with Verma quotients $\Delta(e)$, $\Delta(t)$ and $\Delta(ts)$,
each with multiplicity $1$. Furthermore, using adjointness we find
from Table~\ref{tab:gmult} that $F_{L(s)}L(sts)$ has top
$L(ts)\oplus L(t)$. Thus, $F_{L(s)}L(sts)$ is a quotient of
$P(ts)\oplus P(t)$. The module $P(ts)\oplus P(t)$ has the following
standard filtration:
\[
	P(ts)\oplus P(t) = \smd{\Delta(ts)}{\Delta(t)}{\Delta(s)}{\Delta(e)}
	\oplus \sms{\Delta(t)}{\Delta(e)}.
\]
It is easy to see that this implies that
\[
	F_{L(s)}L(sts) = \Delta(ts)\oplus P(t).
\]
By symmetry, $F_{L(t)}L(sts) = \Delta(st)\oplus P(s)$, which completes
the table.

\vspace{1cm}

\noindent
Department of Mathematics, Uppsala University, Box 480,\\
\mbox{SE-75106 Uppsala}, Sweden.
e-mail: {\tt johan.kahrstrom@math.uu.se}.

\end{document}